\documentclass[preprint,11pt]{elsarticle}

\usepackage{amssymb}
\usepackage{stmaryrd}
\usepackage{tipa}
\usepackage{amsfonts,amsmath,latexsym}
\usepackage{mathrsfs}  
\usepackage{amsbsy}
\usepackage{indentfirst}
\usepackage{amsmath}

\numberwithin{equation}{section}

\newtheorem{theorem}{\textbf{Theorem}}[section]
\newtheorem{lemma}{\textbf{Lemma}}

\newtheorem{remark}[theorem]{Remark}

\usepackage{palatino,eulervm}

\newcommand{\beqnar}{\begin{eqnarray*}}
	\newcommand{\eeqnar}{\end{eqnarray*}}
\newcommand{\ba}{\begin{array}}
	\newcommand{\ea}{\end{array}}

\newenvironment{proof}[1]{\begin{trivlist}\item {\it
			\bf Proof.}\quad} {\qed\end{trivlist}}

\allowdisplaybreaks[4]

\journal{}

\begin{document}
	
	\begin{frontmatter}

		\title{Non-uniform Berry-Esseen Bound by Unbounded Exchangeable Pair Approach }

		\author{Dali Liu}
		
		\address{ Institute for Financial Studies, Shandong University, Jinan, 250100, China
		}
		
		\author{ Zheng Li}
		
		\address{School of Mathematics, Shandong University, Jinan, 250100, China
		}
		
		\author{ Hanchao Wang\footnote{Corresponding author, wanghanchao@sdu.edu.cn}}
		
		\address{Institute for Financial Studies, Shandong University, Jinan, 250100, China
		}

		\author{Zengjing Chen}
		
		\address{School of Mathematics, Shandong University, Jinan, 250100, China
		}		
		
		\begin{abstract}
		In this paper, a new technique is introduced to obtain non-uniform Berry-Esseen bounds for normal and nonnormal approximations by unbounded exchangeable pairs. This technique does not rely on the concentration inequalities developed by Chen and Shao \cite{cls1, cls2} and can be applied to the quadratic forms, the general Curie-Weiss model and an independence test.  In particular, our non-uniform result about the independence test is under 6th moment condition, while the uniform bound in Chen and Shao \cite{cs2}  requires  24th moment condition.
		\end{abstract}
		
		
	\end{frontmatter}
	
	
\section{Introduction}

Since Charles Stein presented his ideas in the seminal paper \cite{s1},  there have been a lot of research activities around Stein's method. Stein's method is a powerful tool to obtaining the approximate error of  normal and non-normal approximations.  The readers are referred to Chatterjee \cite{c} for  recent developments of Stein's method.

While several works on Stein's method pay attention to the uniform error bounds, Stein's method  showed to be  powerful on the non-uniform error bounds, too. By Stein's method, Chen and Shao \cite{cls1, cls2} obtained the non-uniform Berry-Esseen bound for  independent or locally dependent random variables. The key point in their works is the  concentration inequality, which also has strong connection with another approach called  the \textit{exchangeable pair approach}. 

The exchangeable pair approach turned out to be an important topic within Stein's method.  Let  $W$ be the random variable under study.  The pair $ (W,W') $ is called an exchangeable pair if $ (W,W')$ and  $(W',W) $ share the same distribution. With $ \Delta=W-W' $, Rinott and Rotar \cite{rr},  Shao and Su \cite{ss} obtained a Berry-Esseen bound of the normal approximation when $\Delta$ is bounded.  If $\Delta$ is unbounded, Chen and Shao \cite{cs2} provided a Berry-Esseen bound and got the optimal rate for an independence test.  The concentration inequality  plays a crucial role  in previous studies, such as Shao and Su \cite{ss} , Chen and Shao \cite{cs2}. Recently, Shao and Zhang \cite{sz}  made a big step for unbounded $\Delta$ and without using the concentration inequality. They obtained a simple bound as seen from the following result. 

\begin{theorem}(Shao and Zhang \cite{sz}) Let $ (W,W') $ be an exchangeable pair, $ \Delta=W-W' $, and the relation
	\begin{equation*}
	E(\Delta|W)=\lambda(W+R),a.s.,
	\end{equation*}	
	holds for some constant $ \lambda\in(0,1) $ and a random variable $ R $ . Then,
	\begin{align}\label{1.1}
	& \sup\limits_{z\in \mathbb{ R}}\big|  P(W\leq z)-\varPhi(z)  \big|  \notag     \\
	& \leq E\big| 1-\dfrac{1}{2\lambda}E(\Delta^2 |W\big)|+ \dfrac{1}{\lambda}E\big|  E\big(\Delta\Delta^*|W\big)\big| +E| R|, 
	\end{align}
	where $ \varPhi(z) $, $ z\in\mathbb{ R} $, is the standard normal ditribution function, $ \Delta^*(W,W') $ is a random variable satisfying $ \Delta^*(W,W')=\Delta^*(W',W) $ and $ \Delta^*\geq|\Delta| $ ~a.s..
\end{theorem}
In this paper,  inspired by the idea of Shao and Zhang \cite{sz}, we extend their results and get the non-uniform Berry-Esseen bound for unbounded exchangeable pairs by combining new techniques.  In addition, Chatterjee and Shao \cite{cs} introduced a new approach for non-normal approximation by Stein's method in the case of bounded $\Delta$ . When $\Delta$ is unbounded, Shao and Zhang \cite{sz} obtained Berry-Esseen bounds for non-normal approximation. In this paper, we extend their result to the non-uniform case. 

The main contribution of this paper is threefold. First,    we introduce a new technique to obtain non-uniform Berry-Esseen bounds  for unbounded exchangeable pairs. Our proof does not rely on the concentration inequality. Second, we present the non-uniform Berry-Esseen bound  for
 non-normal approximation. As far as we know, there are only a few results in this area. For example, Shao, Zhang and zhang \cite{szz} obtained a Cram\'er-type moderate deviation  for non-normal approximation. At last, we apply our results to quadratic forms, the general Curie-Weiss model, and an independence test.  Especially, our result on  independence test is established under the 6th moment condition, while Chen and Shao \cite{cs2} obtained the same rate for the uniform case under the 24th moment condition.
  
The paper is organized as follows. We present the main result in Section 2.  We give some technical lemmas and the proof of the main result  in Section 3.  The applications of our result are collected  in Section 4.

\section{Main result}

In this section, we  present  some notions and notations about Stein's method. Further details can be found in Shao and Zhang \cite{sz}. We then state our main result.

Let  the function $ g(x),x\in\mathbb{ R} $, of the class $ \mathcal{C}^{2}$, satisfy the following conditions:      
\begin{enumerate}
	\item[(A1)]   $ g(x) $ is non-decreasing, and $ x g(x)\geq0 $ for $ x\in \mathbb{R} $;            
	\item[(A2)] $ g'(x)$ is continuous and $2(g'(x))^{2}-g(x)g''(x)\geq0 $ for all $ x\in \mathbb{R} $;
	\item[(A3)] 
	$ \lim_{x\downarrow -\infty}g(x)p(x)=0$ and $\lim_{x\uparrow +\infty}g(x)p(x)=0 $ , where
	\begin{equation}\label{2.1}
		p(x)=c_1 e^{-G(x)},\quad G(x)=\int^{x}_{0}g(t)dt,~x\in\mathbb{ R},
	\end{equation}       	
	and $ c_1 $ is the constant such that $ \int^{+\infty}_{-\infty}p(x)dx=1 $ .
\end{enumerate} 
Let us  note that if $ g(x)=x $, then $ p(x),x\in \mathbb{R} $, is the standard normal density function. \\

Let $ F(z), z\in\mathbb{ R} $, be the distribution function whose density function is $ p(z) $  as defined in (\ref{2.1}). For a fixed $ z \in \mathbb{ R} $, let $ f_z(x) $ denote the solution of Stein's equation, here and below $ f'_z(x)=\frac{d}{dx}f_z(x) $:
\begin{equation*}
	f_z'(x)-g(x)f_z(x)=I(x\leq z)-F(z), ~x\in\mathbb{ R},
\end{equation*}\\
$ I(\cdot) $ is the indicator function. \\
By Chatterjee and Shao \cite{cs}, 
\begin{equation*}
	f_z(x)=\begin{cases} 
		\dfrac{F(x)(1-F(z))}{p(x)},x\leq z,\\
		\dfrac{F(z)(1-F(x))}{p(x)},x>z.
	\end{cases}
\end{equation*}\\

From Shao and Zhang \cite{sz},  we know that if (A1)$ \sim $(A3) hold, then, for any fixed $ z\in\mathbb{ R} $, $ f_z(x) $ has the following properties:
\begin{enumerate}
	\item[(B1)] $ 0\leq f_z(x)\leq \dfrac{1}{c_1} $ , $ x\in\mathbb{ R}$,
	\item[(B2)] $ \rVert f'_z\rVert\leq1 $  ($ \rVert\cdot\rVert$ is the sup norm.) ;
	\item[(B3)] $ F(z)-1 \leq g(x) f_z(x) \leq  F(z) $  ;
	\item[(B4)] $ g(x)f_z(x) $ is non-decreasing in x.
\end{enumerate}

For a random variable $ W $, applying  Stein's equation to it and taking expectation on both sides, we have:
\begin{equation*}
	P(W\leq z)-F(z)=Ef_z'(W)-Eg(W)f_z(W),\quad z\in\mathbb{ R}.
\end{equation*}

Before presenting our main result, we introduce another  condition  we want $ g(x) $ to  satisfy:
\begin{enumerate}
	\item[(A4)]
	There is  a number $ \tau\in(0,1) $ and a positive constant $ K_{\tau} $ such that  
	\begin{equation*}
		\frac{g(x)}{g(\tau x)}\leq K_{\tau},~ \text{for all $x\in \mathbb{ R} $. }
	\end{equation*}
\end{enumerate}
There is a  large class of functions g satisfying condition $ (A4) $, besides the conditions (A1)$ \sim $(A3). A typical example is $ g(x)=sgn(x)|x|^{\alpha}, \alpha\geq1 $($\alpha$ is a real number).\\

Let \textbf{X} be a random vector in $ \mathbb{ R}^n $ and  $ W=\varphi (\textbf{X}) $  the random variable of interest($\varphi$ is a cartain measurable function). Denote by $ F(z) $, $ z\in\mathbb{ R} $,  the distribution function whose density function is defined by (\ref{2.1}). Now we present our main result.
\begin{theorem}\label{mr}
	Let $ (W,W') $ be an exchangeable pair, $ \Delta=W-W' $, and let  the following relation be satisfied
	\begin{equation}\label{2.2}
		E(\Delta |  \textbf{X})=\lambda(g(W)+R)\quad a.s.,
	\end{equation} for some constant $ \lambda\in(0,1)$ and a random variable R. Assume that $ g(x) $, $ x\in\mathbb{ R} $, satisfies (A1) $ \sim $(A4) and $ Eg^2(W)< \infty $. 
	Then, for any $ z\in\mathbb{ R} $,
	\begin{align}\label{2.3}
		\big| P(W\leq z)-F(z)\big|\leq\dfrac{C}{1+|g(z)|}\Big\{ \sqrt{E\big|(1-\dfrac{1}{2\lambda}E(\Delta^{2}| \textbf{X})) \big|^{2} } 
		+\frac{1}{\lambda}\sqrt{E| E(\Delta\Delta^*| \textbf{X})|^{2}}+E| R|\Big\}.
	\end{align}       	
	Here C is a constant depending on $ \tau $ and $ Eg^2(W) $, $ \Delta^* $ is a random variable such that $ \Delta^*(W,W')=\Delta^*(W',W) $ and $ \Delta^*\geq |\Delta| $ a.s..
\end{theorem}

\begin{remark}
	Shao and Zhang \cite{sz} provided the Berry-Esseen bound for non-normal approximation similar to (\ref{1.1}). Theorem \ref{mr} is a non-uniform refinement of their result.
\end{remark}
\begin{remark}
	Let $ W=\sum\limits_{i=1}^n X_i $, where $ \{X_i, i=1,...n\} $ are independent random variables with zero mean and $ EW^2=1 $. Our general result (\ref{2.3}) cannot directly cover the following classical result in Chen and Shao \cite{cls1}:  there is an absolute constant C such that for any $ z\in\mathbb{ R} $,
	\begin{equation}\label{2.4}
		\Big|P(W\leq z)-\Phi(z)\Big|\leq C \sum\limits_{i=1}^n\Big\{ \frac{EX_i^2I(|X_i|>1+|z|)}{(1+|z|)^2} +\frac{E|X_i|^3I(|X_i|\leq 1+|z|)}{(1+|z|)^3}   \Big\}.
	\end{equation}
	In the above, $ \Phi(z) $, $ z\in\mathbb{ R} $ is the standard normal distribution. However, the technique "leave one out" to deal with sums of independent variables is very similar to the exchangeable pair technique. If we begin with (\ref{B}) (see in the proof of the main result)  and use some results from Chen and Shao \cite{cls1}, it is not difficult to obtain (\ref{2.4}) . In some applications such as the quadratic forms treated later in this paper, where  relation (\ref{2.2}) is satisfied with $ R=0 $ and $ g(x)=x $,  the non-uniform part $ \frac{C}{1+|z|} $ in (\ref{2.3}) can be improved significantly by replacing it with $ \frac{C}{(1+|z|)^2} $.
\end{remark}

\section{Proof of Theorem  \ref{mr}.}

In what follows, C is used to denote a constant whose value may change at each occurrence. 

Since  $ (W,W') $ and $ (W',W) $ have the same distribution and $ E(\Delta|\textbf{X})=\lambda(g(W)+R) $,  following the same arguments as in Shao and Zhang \cite{sz}, we obtain
\begin{align*}
	0&=E(W-W')(f_z(W)+f_z(W')\\
	&=E(W-W')(2f_z(W)+f_z(W')-f_z(W))\\
	&=2\lambda E(g(W)f_z(W))+2\lambda Ef_z(W)R- E\Delta\int^0_{-\Delta}f_z'(W+t)dt.
\end{align*}
Thus
$Eg(W)f_z(W)=\frac{1}{2\lambda}E\Delta\int^0_{-\Delta}f_z'(W+t)dt-Ef_z(W)R. $ 
Then
\begin{align*}
	&Ef_z'(W)-Eg(W)f_z(W)\\
	&=Ef_z'(W)-\frac{1}{2\lambda}E\Delta\int^0_{\Delta}f_z'(W+t)dt+Ef_z(W)R\\
	&=Ef_z'(W)\Big(1-\frac{1}{2\lambda}\Delta^2\Big)-\frac{1}{2\lambda}E\Delta\int^0_{-\Delta}f_z'(W+t)-f_z'(W)dt+Ef_z(W)R\\
	&=Ef_z'(W)\Big(1-\frac{1}{2\lambda}E(\Delta^2|\textbf{X})\Big)-\frac{1}{2\lambda}E\Delta\int^0_{-\Delta}f_z'(W+t)-f_z'(W)dt+Ef_z(W)R.\\
\end{align*}
With the notation $ J=\frac{1}{2\lambda}E\Delta\int^0_{-\Delta}f_z^{'}(W+t)-f_z^{'}(W)dt $, we find that
\begin{align*}
	J&=\frac{1}{2\lambda}E\Big(\Delta\int^0_{-\Delta}g(W+t)f_z(W+t)-g(W)f_z(W)dt \Big)\\
	&+\frac{1}{2\lambda}E\Big(\Delta\int^0_{-\Delta}I(W+t\leq z)-I(W\leq z)dt\Big)\\
	&=J_{1}+J_{2},
\end{align*}
where 
\begin{align*}
	&J_{1}=\frac{1}{2\lambda}E\Big(\Delta\int^0_{-\Delta}g(W+t)f_z(W+t)-g(W)f_z(W)dt \Big),\\
	&J_{2}=\frac{1}{2\lambda}E\Big(\Delta\int^0_{-\Delta}I(W+t\leq z)-I(W\leq z)dt\Big).
\end{align*}
From Shao and Zhang \cite{sz}, it is known that
\begin{align}\label{I1}
	|J_{1}|\leq\frac{1}{2\lambda}E\Delta\Delta^*g(W)f_z(W)	
\end{align}	
and
\begin{align}\label{I21}
	|J_{2}|\leq\frac{1}{2\lambda}E\Delta\Delta^*I(W>z).
\end{align}
Observe that
\begin{align*}
	E\Delta\Delta^*=0.
\end{align*}
Then we obtain
\begin{align}\label{I22}
	|J_{2}|\leq\frac{1}{2\lambda}E\Delta\Delta^*\big(I(W>z)-1\big)=\frac{1}{2\lambda}E\Delta\Delta^*I(W'\leq z).
\end{align}
Combining (\ref{I1}) and (\ref{I21}), for $ z> 0 $, we have 
\begin{align}\label{B}
	&\big| P(W\leq z)-F(z) \big | \notag\\
	&\leq E\Big|f_z^{'}(W)\Big(1-\dfrac{1}{2\lambda}E(\Delta^{2}|\textbf{X})\Big) \Big|
	+\dfrac{1}{2\lambda}E\Big| g(W)f_z(W)E(\Delta\Delta^*|\textbf{X})\Big| \notag\\
	& 	+\dfrac{1}{2\lambda}E\big| E(\Delta\Delta^*|\textbf{X})I(W> z)\big|+E| f_z(W)R|.
\end{align}
For $ z\leq 0 $, using (\ref{I1}) and (\ref{I22}), we have
\begin{align}\label{C}
	&\big| P(W\leq z)-F(z) \big | \notag\\
	&\leq E\Big|f_z^{'}(W)\Big(1-\dfrac{1}{2\lambda}E(\Delta^{2}|\textbf{X})\Big) \Big|
	+\dfrac{1}{2\lambda}E\Big| g(W)f_z(W)E(\Delta\Delta^*|\textbf{X})\Big| \notag\\
	& 	+\dfrac{1}{2\lambda}E\big| E(\Delta\Delta^*|\textbf{X})I(W'\leq z)\big|+E| f_z(W)R|.
\end{align}
The only difference between (\ref{B}) and (\ref{C}) is that the expectation $ E\big| E(\Delta\Delta^*|\textbf{X})I(W> z)\big| $ is replaced by $ E\big| E(\Delta\Delta^*|\textbf{X})I(W'\leq z)\big| $.\\
To prove (\ref{2.3}), we first assume that $ z> 0 $.\\
Cauchy's inequality applied to the fist term of (\ref{B}) yields
\begin{equation}\label{B1a}
	E\Big|f_z'(W)\Big(1-\dfrac{1}{2\lambda}E(\Delta^{2}|\textbf{X})\Big) \Big|\leq 
	\sqrt{E| f_z^{'}(W)|^{2}}\cdot
	\sqrt{E\Big|(1-\dfrac{1}{2\lambda}E(\Delta^{2}\mid \textbf{X})) \Big|^{2} }.
\end{equation}
We will show   now that
\begin{equation}\label{B1b}
	\sqrt{E| f_z'(W)|^{2}}\leq\dfrac{C}{1+| g(z)|} .
\end{equation}
Indeed, for any $ \tau\in(0,1) $, we have 
\begin{equation*}
	E| f_z'(W)|^2=E| f_z'(W)|^2 I(W\leq0)+E|f_z'(W)|^2 I(0<W\leq\tau z)+E| f_z'(W)|^2 I(W>\tau z).
\end{equation*}
Recall that for $ x\leq 0 $, 
\begin{equation*}
	f'_z(x)=g(x)f_z(x)+1-F(z)=\Big(\frac{F(x)g(x)}{p(x)}+1\Big)\cdot(1-F(z)) .
\end{equation*} 
Because $ g(x)f_z(x) $ is increasing in $ x\in(-\infty,0) $ and for fixed $ z $, $ F(z)-1\leq g(x)f_z(x)\leq F(z) $, we see that
\begin{equation*}
	-1\leq\frac{F(x)g(x)}{p(x)}\leq\frac{F(0)g(0)}{p(0)}=0 .
\end{equation*}
Thus we conclude that $ \frac{F(x)g(x)}{p(x)}+1 $ is bounded on $ (-\infty,0) $ and it does not  depend on z. We notice further that 
\begin{equation}\label{a1}
	1-F(z)=\int_{z}^{\infty}p(y)dy\leq \int_{z}^{\infty}\frac{g(y)}{g(z)}p(y)dy\leq \frac{p(z)}{g(z)}.
\end{equation} Hence
\begin{equation}\label{a2}
	E| f_z'(W)|^2 I(W\leq0) \leq C(1-F(z))^2\leq C\left(\dfrac{p(z)}{g(z)}\right) ^2\leq\dfrac{C}{g^2(z)}. 
\end{equation} 
By (\ref{a1}), we have
\begin{align*}
	E| f_z'(W)|^2 I(0<W\leq\tau z)&= EI(0<W\leq\tau z)\cdot(1+\frac{1}{c_1}F(W)\cdot g(W)e^{G(W)})\cdot(1-F(z))^2\\
	&\leq C\Big(1+g(\tau z)exp\Big({\int^{\tau z}_0g(y)dy}\Big)\Big)^2\cdot exp\Big(-2\int^z_0g(y)dy\Big)\cdot\frac{1}{g^2(z)}.
\end{align*}
We notice that 	
\begin{align*}
	g(\tau z)exp\left({\int^{\tau z}_0g(y)dy}\right)\cdot exp\left({-\int^z_0g(y)dy}\right)&=g(\tau z)exp\left({-\int^{z}_{\tau z}g(y)dy}\right)\\
	&\leq g(\tau z)e^{-(1-\tau)zg(\tau z)}\\
	&\leq C .
\end{align*}
Therefore $E| f_z'(W)|^2 I(0<W\leq\tau z)	\leq	\dfrac{C}{g^2(z)} $ and C depends on $ \tau. $\\
For the  term $  E| f_z'(W)|^2 I(W>\tau z)$, by (B3) and (A4), we find, by Markov's inequality, that
\begin{align*}
	E| f_z'(W)|^2 I(W>\tau z)&\leq P(W>\tau z)\\
	&\leq\dfrac{Eg^2(W)}{g^2(\tau z)}\\
	&\leq \frac{C}{g^2(\tau z)}=\frac{g^2(z)}{g^2(\tau z)}\cdot\frac{C}{g^2(z)}\\
	&\leq \frac{K^2_{\tau}\cdot C}{g^2(z)}.
\end{align*}
Thus $ E| f_z'(W)|^2\leq \frac{C}{g^2(z)}$ for $ z> 0 $ with a constant C depending on $ \tau $ and $ Eg^2(W) $.\\    
The next is to use the fact that $ \rVert f^{'}_z\rVert \leq 1 $ and see that   $$ \sqrt{E\mid f_z'(W)\mid^{2}}\leq \min\Big\{1,~\dfrac{C}{|g(z)|}\Big\}\le  \dfrac{C}{1+| g(z)|} ,$$
which complete the proof of (\ref{B1b})  for $ z>0 $. By (\ref{B1a}) and (\ref{B1b}), we have
\begin{equation}\label{B1c}
	E\Big|f_z'(W)\Big(1-\dfrac{1}{2\lambda}E(\Delta^{2}|\textbf{X})\Big) \Big|\leq 
	\dfrac{C}{1+|g(z)|}\cdot
	\sqrt{E\Big|(1-\dfrac{1}{2\lambda}E(\Delta^{2}\mid \textbf{X})) \Big|^{2} }.
\end{equation}\\
Using Cauchy's inequality, for the second term of (\ref{B}), we find 
\begin{equation}\label{B2a}
	\dfrac{1}{2\lambda}E\Big| g(W)f_z(W)E(\Delta\Delta^*|\textbf{X})\Big|\leq
	\sqrt{E|g(W)f_z(W)|^{2}}\cdot
	\frac{1}{2\lambda}\sqrt{E| E(\Delta\Delta^*\mid  \textbf{X})|^{2}}.
\end{equation}
We will show that
\begin{equation}\label{B2b}
	\sqrt{E|g(W)f_z(W)|^{2}}\leq\dfrac{C}{1+|g(z)|}.
\end{equation}
Since we know that $ g(x)f_z(x)=f_z'(x)-\big(I(x\leq z)-F(z)\big) $ , $ \rVert gf_z\rVert\leq 1 $ and $ E| f_z'(W)|^2\leq \frac{C}{g^2(z)}$, we only need to show  that $ E\big(I(W\leq z)-F(z)\big)^2\leq \frac{C}{g^2(z)} $.\\
For $ z>0 $,
\begin{align*}
	E\big(I(W\leq z)-F(z)\big)^2=E\big(1-F(z)\big)^2I(W\leq z)+F^2(z)I(W>z)\leq \frac{C}{g^2(z)}.
\end{align*}
Thus we have proved (\ref{B2a}) for $ z> 0 $. By (\ref{B2a}) and (\ref{B2b}), we have
\begin{equation}\label{B2c}
	\dfrac{1}{2\lambda}E\Big| g(W)f_z(W)E(\Delta\Delta^*|\textbf{X})\Big|\leq
	\dfrac{C}{1+|g(z)|}\cdot
	\frac{1}{2\lambda}\sqrt{E| E(\Delta\Delta^*\mid  \textbf{X})|^{2}}.
\end{equation}\\
For the third term of (\ref{B}), we obtain
\begin{equation}\label{B3a}
	\dfrac{1}{2\lambda}E\big| E(\Delta\Delta^*|\textbf{X})I(W> z)\big|\leq\sqrt{P(W> z)}\sqrt{\dfrac{1}{2\lambda}E\big| E(\Delta\Delta^*|\textbf{X})\big|^2}.
\end{equation}
By Markov's inequality,
\begin{align*}\label{4.13}
	\notag P(W> z)&\leq\dfrac{Eg^2(W)}{g^2( z)}\\
	&\leq  \frac{C}{g^2(z)}.
\end{align*}
Then , (\ref{B3a}) becomes
\begin{equation}\label{B3b}
	\dfrac{1}{2\lambda}E\big| E(\Delta\Delta^*|\textbf{X})I(W> z)\big|\leq\dfrac{C}{1+|g(z)|}\sqrt{\dfrac{1}{2\lambda}E\big| E(\Delta\Delta^*|\textbf{X})\big|^2}.
\end{equation}
From Shao, Zhang and Zhang \cite{szz}, we know that
\begin{equation}\label{B4a}
	\Vert f_z\Vert\leq \min\Big\{ \frac{1}{c_1}, ~\frac{1}{|g(z)|}\Big\}
\end{equation}
for $ z\in\mathbb{ R} $. For the last term of ( \ref{B}), we have
\begin{equation}\label{B4b}
	E|f_z(W)R|\leq \dfrac{C}{1+|g(z)|}E|R|.
\end{equation}
From (\ref{B}), (\ref{B1c}), (\ref{B2c}), (\ref{B3b}) and (\ref{B4b}), it follows that we have proved (\ref{2.3}) for $ z> 0 $.\\
For  $ z\leq0 $, we take (\ref{C}) and use Cauchy's inequality. For the third term of ( \ref{C}), it is easy to see that
\begin{align}\label{C1}
	\notag\dfrac{1}{2\lambda}E\big| E(\Delta\Delta^*|\textbf{X})I(W'\leq z)\big|&\leq\sqrt{P(W'\leq z)}\sqrt{\dfrac{1}{2\lambda}E\big| E(\Delta\Delta^*|\textbf{X})\big|^2}
	\\
	&\leq\dfrac{C}{1+|g(z)|}\sqrt{\dfrac{1}{2\lambda}E\big| E(\Delta\Delta^*|\textbf{X})\big|^2}.
\end{align}
For the last term of (\ref{C}), in view of  (\ref{B4a}),
\begin{equation}\label{C2}
	E|f_z(W)R|\leq \dfrac{C}{1+|g(z)|}E|R|.
\end{equation}
Thus we only need to prove (\ref{B1b}) and  (\ref{B2b}) for $ z\leq 0 $.\\
For $ z\leq0 $, we have, for any $ \tau\in(0,1) $, that
\begin{equation*}
	E| f_z'(W)|^2=E| f_z'(W)|^2 I(W\leq\tau z)+E| f_z'(W)|^2 I(\tau z\leq W\leq 0)+E| f_z'(W)|^2 I(W>0).
\end{equation*} 
By the same arguments as above,  we obtain
\begin{equation}\label{C3}
	F(z)\leq\frac{p(z)}{|g(z)|},\quad z\leq 0.
\end{equation}
Then following  similar steps as in the proof for $ z\geq 0 $, we establish (\ref{B1b}) for $ z\leq 0$. To prove (\ref{B2b}) for $ z\leq 0 $, it suffices to show that $ E\big(I(W\leq z)-F(z)\big)^2\leq C/|g(z)|^2 $ for $ z\leq0 $. Indeed,  by (\ref{C3}) and Markov's inequality,
\begin{align*}
	E\big(I(W\leq z)-F(z)\big)^2\leq 2P(W\leq z)+2F^2(z)\leq\frac{C}{|g(z)|^2}.
\end{align*}
Let us summarize our findings:(\ref{B}), (\ref{B1c}), (\ref{B2c}), (\ref{B3b}) and (\ref{B4b})  show that the bound (\ref{2.3}) is true for $ z>0 $, while (\ref{B1b}), (\ref{B2b}) proved for $ z\leq 0 $, and (\ref{C2}), (\ref{C3}) show that this bound holds for $ z\leq 0 $.

Theorem \ref{mr} is proved.

\section{Applications}
\subsection{Quadratic forms} 

Let $ X_1,X_2\cdots,X_n $ be i.i.d. random variables with  zero mean, unit variance and a finite fourth moment. Let $ A=(a_{ij})_{1\le i, j\le n} $ be a real symmetric matrix with $ a_{ii}=0 $ and let 
\begin{align*}
	W_n=\dfrac{1}{\sigma_n}\sum_{i\neq j}a_{ij}X_{i}X_{j},~~ \sigma^2_n=2\sum_{i=1}^{n}\sum_{j=1}^{n}a^2_{ij}.
\end{align*}This is a classical example which has been widely discussed in the literature. For example, de Jong \cite{dj} obtained the asymptotic normality of $ W_n $,  Chatterjee \cite{c1} gave an $ L^1 $ bound and  
G\"{o}tze and Tikhomirov \cite{gt} studied the Kolmogorov distance between the distribution  of $ W_n $ and the distribution of the same quadratic forms with $ X_{ij} $ repalced by corresponding Gaussian random variables. Shao and Zhang \cite{sz} established the following bound:
\begin{equation*}
	\sup\limits_{z\in\mathbb{ R}}\big| P(W_n\leq z)-\Phi(z) \big|\leq 
	\frac{CEX_1^4}{\sigma_n^2}
	\left(\sqrt{\sum\limits_i\Big(\sum\limits_j a_{ij}^2\Big)^2}+\sqrt{\sum\limits_{i,j}\Big(\sum\limits_k (a_{ik} a_{jk})^2\Big)}\right).
\end{equation*} 
The next theorem is a non-uniform refinement of this bound.
\begin{theorem}\label{th4.1}
	Let $ \{X_1,X_2,...,X_n \}$ be i.i.d random variables with zero mean, unit variance and a finite fourth moment. Let $ A=(a_{ij})^n_{i,j=1} $ be a real symmetric matrix with $ a_{ii}=0 $ for all $ 1\leq i\leq n $.  Put $ W_n=\dfrac{1}{\sigma_n}\sum_{i\neq j}a_{ij}X_{i}X_{j} $ and $ \sigma^2_n=2\sum_{i=1}^{n}\sum_{j=1}^{n}a^2_{ij} $. Then,
	\begin{equation}\label{3.1}
		\big| P(W_n\leq z)-\Phi(z) \big|\leq\dfrac{CEX^4_1}{(1+| z|)^2\sigma_n^2}\left(\sqrt{\sum\limits_i\Big(\sum\limits_j a_{ij}^2\Big)^2}+\sqrt{\sum\limits_{i,j}\Big(\sum\limits_k (a_{ik} a_{jk})^2\Big)}\right),
	\end{equation}
	where C is an absolute constant depending on $ EX_1^4 $.
\end{theorem}

\begin{proof}
	
	Let $ (X'_1,X'_2,...,X'_n) $ be an independent copy of $ (X_1,X_2,...,X_n) $ and $ \theta$ a disrete uniformly distributed random variable over the set $ \{1,2,...,n\} $ and independent of all  oher random variables.  Define 
	\begin{equation*}
		W'_n=W_n-\frac{2}{\sigma_n}\sum_{j=1}^{n}a_{\theta j}X_\theta X_j+ \frac{2}{\sigma_n}\sum_{j=1}^{n}a_{\theta j}X'_{\theta}X_j.
	\end{equation*}Then $ (W,W') $ is an exhcangeable pair. It is easy to see that
	\begin{equation*}
		\Delta=W_n-W_n'=\frac{2}{\sigma_n}\sum\limits_{i=1}^nI\{\theta=i\}\sum\limits^n_{j=i}a_{ij}X_j(X_i-X'_i)
	\end{equation*}
	and 
	\begin{equation*}
		E(\Delta|\textbf{X})=\frac{2}{n}W_n.
	\end{equation*}
	These relations imply that condition (\ref{2.2}) is satisfied with $ g(x)=x$, $ \lambda=\frac{2}{n} $ and $ R=0 $. 
	By Shao and Zhang \cite{sz},
	\begin{equation}\label{q1}
		E\Big|1-\frac{1}{2\lambda}E(\Delta^2|\textbf{X})\Big|^2 \leq C\sigma_n^{-4}\Big(E(X_1^4)\Big)^2\left(\sum\limits^n_{i=1}\Big(\sum\limits^n_{j=1} a_{ij}^2\Big)^2+\sum\limits^n\limits_{i,j=1}\Big(\sum\limits^n\limits_{k=1} (a_{ik} a_{jk})^2\Big)\right)
	\end{equation}
	and
	\begin{equation}\label{q2}
		Var\Big(\frac{1}{\lambda}E(\Delta|\Delta||\textbf{X})\Big)\leq C\sigma^{-4}_nE^2(X_1^4)\sum\limits_i^n\Big(\sum\limits_j^n a^2_{ij}\Big)^2  .
	\end{equation}
	
	Note that  $ EX_1^4<\infty $ and  $ EW_n^4<C $ for any $ n=1,2 \cdots.$ Then, 
	\begin{equation*}
		P(|W_n|>z)\leq \frac{EW_n^4}{z^4}\land 1=\min\{1,~C/z^4\}\leq \frac{C}{(1+|z|)^4}.
	\end{equation*} 
	For  $ \tau $ involved in (A4), we can take, for example  $ \tau=\frac{1}{2} $ and derive that 
	\begin{align*}
		E| f_z'(W_n)|^2 I(0<W\leq\frac{1}{2} z)&\leq \Big[C\big(1+ze^{\int^{ z/2}_0 ydy}\big)^2\cdot e^{-2\int^z_0ydy}\cdot \frac{1}{z^2}\Big]\land 1\\
		&\leq \Big[C\big( e^{-z^2}/z^2+ e^{-3z^2/4}\big)\Big]\land 1\\
		&=\Big[C \big(z^2e^{-z^2}/z^4 + z^4e^{-3z^2/4}/z^4 \big)\Big]\land 1\\
		&\leq\min\{1, ~C/z^4\}\\
		&\leq \frac{C}{(1+|z|)^4}.
	\end{align*}
	By( $ \ref{a2} $),  we have
	\begin{align*}
		E|f'_z(W_n)|^2I(W\leq 0)\leq C(\frac{p(z)}{g(z)})^2=Ce^{-z^2/2}/z^2=Cz^2e^{-z^2/2}/z^4\leq C/z^4.
	\end{align*}
	Then, 
	\begin{equation*}
		E|f'_z(W_n)|^2\leq \frac{C}{(1+|z|)^4}.
	\end{equation*}
	Using the same arguments as those in the proof of the main result, we find that
	\begin{align*}
		\sqrt{E|f'_z(W_n)|^{2}}\leq\dfrac{C}{(1+| z|)^2},\\
		\sqrt{E|W_nf_z(W_n)|^{2}}\leq\dfrac{C}{(1+|z|)^2}.
	\end{align*}
	Hence the bound $ \frac{C}{1+|z|} $ in (\ref{2.3}) can be improved replacing it by $ \frac{C}{(1+|z|)^2} $. Thus, referring to Theorem $\ref{mr}$, in view of (\ref{q1}) and (\ref{q2}), we complete the proof of  this theorem.

\end{proof}

\subsection{General Curie-Weiss model}

The Curie-Weiss model is important in statistical physics and has been extensively discussed in the literature. For some history and the first asymptotic results, the reader is referred to Ellis and Newman \cite{e1}, \cite{e2}.  Using exchangeable pairs,  Chatterjee and Shao \cite{cs} studied  a kind of Curie-Weiss model. Shao and Zhang \cite{sz} studied a general Curie-Weiss model and got the optimal convergence rate. In this subsection, we refine the bound in Shao and Zhang \cite{sz} to the non-uniform case.       

Let $ L(x) $,$ x\in\mathbb{ R} $, be a distribution function satisfying the conditions:

\begin{equation}\label{cw1}
	\int^{+\infty}_{-\infty}xdL(x)=0  \quad\text{ and } \quad \int^{+\infty}_{-\infty}x^2dL(x)=1 
\end{equation}
For a positive integer k and a real number $ \lambda_{\rho} $, say that $L$ is of type $k$ with strength $ \lambda_{\rho} $, if
\begin{equation*}
	\int^{+\infty}_{-\infty}x^{j}d\Phi(x)-\int^{+\infty}_{-\infty}x^jdL(x)=\begin{cases}
		0,  \quad\text{for}~~j=1,\cdots, 2k-1,\\
		\lambda_{\rho}, \quad\text{for}~~j=2k,
	\end{cases}
\end{equation*}
where, to recall that $ \Phi(x) $,$ x\in\mathbb{ R} $, is the standard normal distribution function. 

Let $ (X_1,\cdots,X_n) $ be a random vector with joint  distribution function $ P_{n,\beta}(\textbf{x}) $, $ \textbf{x}=(x_1,\cdots,x_n)\in\mathbb{ R}^n $ such that
\begin{equation}\label{cw2}
	dP_{n,\beta}(\textbf{x})=\frac{1}{K_n}\exp\Big(\dfrac{\beta(x_1+...+x_n)^2}{2n}\Big)\prod\limits^{n}_{i=1}dL(x_i) 
\end{equation}
where  $ K_n $ is the normalizing constant. Let $ \xi $ be a random variable with distribution function $ L $. Moreover, assume that:
\begin{enumerate}
	\item[(1)]  for $ 0<\beta<1 $, there exists a constant $ b>\beta $ such that
	\begin{equation}\label{cw3}
		Ee^{t\xi}\leq e^{t^2/2b}, ~~t\in\mathbb{ R}.
	\end{equation}
	\item[(2)]  for $ \beta=1 $, there exist constants $ b_0>0 $, $ b_1>0 $ and $ b_2>1 $ such that:
	\begin{equation}\label{cw4}
		Ee^{t\xi}\leq\begin{cases}
			\exp(t^2/2-b_1t^{2k}),\quad| t|\leq b_0,\\
			\exp(t^2/2b_2),\quad|t|>b_0.
		\end{cases}
	\end{equation}
\end{enumerate}              
We have the following results:
\begin{theorem}
	Suppose that  the distribution function of the random vector $ (X_1,X_2,\cdots,X_n) $ is given by (\ref{cw2}), where $ L $ satisfies (\ref{cw1}) and let $ S_n=X_1+\cdots+X_n $.\\ 
	$ (i) $ If $ 0<\beta<1 $ and (\ref{cw3}) is satisfied, $ W_n=S_n/\sqrt{n} $. Then
	\begin{equation}\label{cw5}
		\big|P(W_n\leq z)-F_1(z)\big|\leq\dfrac{C}{1+(1-\beta)| z|}\cdot\frac{1}{\sqrt{n}} ,
	\end{equation}
	where $ F_1(z) $, $ z\in\mathbb{ R} $, is the distribution function of  a random variable $ Z_1\sim $ $ \cal{N} $$(0,\frac{1}{1-\beta} )$ and C is a constant depending on $ b  $ and $ \beta $.\\
	$ (ii) $ If $ \beta=1,L $ is of type k, (\ref{cw4}) holds and $ W_n=S_n/n^{1-1/2k} $, then	
	\begin{equation}\label{cw6}
		\big| P(W_n\leq z)-F_k(z)\big|\leq\dfrac{C}{1+2kc_2|z|^{2k-1}}\cdot\frac{1}{n^{1/2k}},
	\end{equation}
	where C is a constant depending on $ b_0,b_1,b_2 $ and k. $ F_k(z) $, $ z\in\mathbb{ R} $, is the distribution function whose density function is $p_k(z)=c_1e^{-c_2y^{2k}}  $, $ c_2=\dfrac{H^{(2k)}(0)}{(2k)!} $ , $ c_1 $ is the normalizing constant, and 
	\begin{align*}
		H(s)=s^2/2-\ln(\int^{+\infty}_{-\infty}exp(sx)dH(x)),~s\in\mathbb{ R}.
	\end{align*}$   $. 	
\end{theorem}

\begin{proof}
	
	Recall that $ S_n=\sum\limits_{i=1}^nX_i $. We first construct an exchangeable pair as follows. For  a fixed i, $ 1\leq i\leq n $, given $ \{X_j,j\neq i\} $, let $ X'_i $ be a random variable which is conditionally independent of $ X_i $ and has the same conditional distribution as $ X_i $. Let $ \theta $ be a random index unformly distributed over $ \{1,\cdots,n\}  $ and independent of all other random variables. Let $ S'_n=S_n-X_\theta+X'_\theta $. Then $ (S_n,S'_n) $ is an exhangeable pair. 
	
	When $  0<\beta<1 $, let $ W_n=S_n/\sqrt{n} $ and $ W'_n=S'_n/\sqrt{n} $. Then $ (W_n,W_n') $ is an exchangeable pair.  By Shao and Zhang \cite{sz},  the following relations are satisfied: 
	\begin{align}
		&E(W_n-W'_n|\textbf{X})=\frac{1}{n}\Big((1-\beta) W_n+\sqrt{n}R_2\Big);\label{cw7}\\
		&E|R_2|\leq C n^{-1/2}; \label{cw8} \\
		&E\Big|\frac{1}{2\lambda}E((W_n-W'_n)^2|\textbf{X})-1\Big|^2\leq Cn^{-1};\label{cw9}\\
		&E\Big|\frac{1}{2\lambda}E((S_n-S'_n)|S_n-S'_n||\textbf{X})\Big|^2\leq n^{-1}.\label{cw10}
	\end{align}
	Here $C$ depends on $ \beta $ and $ b $. Thus (\ref{2.2}) is satisfied with $ g(x) =(1-\beta)x$, and $ \lambda=\frac{1}{n} $. Using (\ref{cw8}), (\ref{cw9}), (\ref{cw10}) and Theorem \ref{mr} , we obtained (\ref{cw5}).
	
	When $ \beta=1 $ , recall that  $ W_n=S_n/n^{1-1/2k} $ and define  $ W'_n=S'_n/n^{1-1/2k} $, so $ (W_n,W_n') $ is an exchangeable pair. By Shao and Zhang \cite{sz},  we obtain the following:
	\begin{align}
		&E(W_n-W'_n|\textbf{X})=n^{-2+1/k}\Big(\frac{H^{(2k)}(0)}{(2k-1)!}W_n^{2k-1}+n^{-1+1/2k}R_1\Big);\notag\\
		&E|R_1|\leq Cn^{-1/2k};\label{cw11}\\
		&E\Big|\frac{1}{2\lambda}E((W_n-W'_n)^2|\textbf{X})-1\Big|^2\leq Cn^{-1/k};\label{cw12}\\
		&E\Big|\frac{1}{2\lambda}E((W_n-W'_n)^2|\textbf{X})-1\Big|^2\leq Cn^{-1}.\label{cw13}
	\end{align}
	Here $C$ depends on $ \beta $ and $ b $. 	Thus $ g(x)=\frac{H^{(2k)}(0)}{(2k-1)!}x^{2k-1}=2kc_2x^{2k-1} $ and $ \lambda=n^{-2+\frac{1}{2k}} $. By (\ref{cw11}),(\ref{cw12}), (\ref{cw13}) and Theorem \ref{mr}, we obtain (\ref{cw6}).
	
\end{proof}
\subsection{ Independence test}

Independence test is a classical problem in statistics.    Consider a $p$-dimensional population represented by a random vector $X=(X_1,X_2,\dots,X_p)'$ with covariance matrix $\sum$ and let $X_i=(X_{i1},X_{i2},\dots,X_{in})$ be a random sample of size $n$ selected from $X_{i}$. Recently, a great attention has been paid to the case of large $p$,  see Bai and Saranadasa \cite{bs},   Fan and Li \cite{fl},  Jiang \cite{j}, Liu, Lin and Shao \cite{lls}, Chen and Liu \cite{cl} and the references theorems.

Chen and Shao \cite{cs2}  studied the following statistics.   Let $R=(r_{ij},1\leq i,j\leq p)$ be the sample correlation matrix, where
\begin{equation*}
r_{ij}=\frac{\sum^{n}_{k=1}(X_{ik}-\bar{X}_{i})(X_{jk}-\bar{X}_{j})}{\sqrt{\sum^{n}_{k=1}(X_{ik}-\bar{X}_{i})^2}\sqrt{\sum^{n}_{k=1}(X_{jk}-\bar{X}_{j})^2}}.
\end{equation*}
With the usual notation $\bar{X}_{i}=\frac{1}{n}\sum^{n}_{k=1}X_{ik}$. Now we define $t_{n,p}$ as follows:
\begin{align*}
t_{n,p}=\sum^p_{i=2}\sum^{i-1}_{j=1}r_{ij}^2,
\end{align*}
and let
\begin{equation*}
W_{n,p}=c_{n,p}(t_{n,p}-\frac{p(p-1)}{2(n-1)}),\quad\text{where}\quad c_{n,p}=\frac{n\sqrt{n+2}}{\sqrt{p(p-1)(n-1)}}.
\end{equation*}

If $ X_{ij} $ are i.i.d. random variables, satisfy the condition $E(X_{11}^{24})<\infty$,  and $p=O(n)$,  Chen and Shao \cite{cs2} obtained the following upper bound:
\begin{equation*}
\sup_{z}|P(W_{n,p}\leq z)-\Phi(z)|=O(p^{-1/2}).
\end{equation*}

Our approach allows us to establish the following result.

\begin{theorem}\label{it}
	Let $\{X_{ij},1\leq i,j\leq p\}$ be i.i.d random variables. Assume that $p=O(n)$ and the condition $E(X_{11}^6)<\infty$ is satisfied. Then
	\begin{equation}\label{itr}
	|P(W_{n,p}\leq z)-\Phi(z)|\leqslant\frac{C}{(1+|z|)^2}p^{-1/2},\quad z\in R
	\end{equation}
	
\end{theorem}

Let $\{X^*_{i} \}$ be an independent copy of $\{ X_{i} \}$ and as before,$\theta$ be a random variable uniformly distributed over $\{1,2,\cdots,p\}$; $\theta$ is independent of all $ \{X_{i},X_{i*}, 1\leq i\leq p  \}$. With $t^{*}_{n,p}=t_{n,p}-\sum^p_{j=1, j\neq \theta}r^2_{\theta j}+\sum^p_{j=1, j\neq \theta}r^2_{\theta^*j}$, where
\begin{equation*}
r_{i^*j}=\frac{\sum^{n}_{k=1}(X^*_{ik}-\bar{X}^*_{i})(X_{jk}-\bar{X}_{j})}{\sqrt{\sum^{n}_{k=1}(X^{*}_{ik}-\bar{X}^*_{i})^2}\sqrt{\sum^{n}_{k=1}(X_{jk}-\bar{X}_{j})^2}}.
\end{equation*}
We define
\begin{equation*}
W_{n,p}^*=c_{n,p}\Big(t_{n,p}^*-\frac{p(p-1)}{2(n-1)}\Big).
\end{equation*}
By Chen and Shao \cite{cs2}, $(W_{n,p},W_{n,p}^*)$ is an exchangeable pair and
$$
E(W_{n,p}-W_{n,p}^*|\textbf{X})=\frac{2}{p}W_{n,p}.
$$
Then (\ref{2.2}) is satisfied with $ g(x)=x $ , $ \lambda=\frac{2}{p} $ and $ R=0 $. To finish the proof, we begin with some premilinary propeties of $W_{n,p}  $.
Denote
\begin{equation*}
u_{ik}=\frac{X_{ik}-\bar{X}_{i}}{\sqrt{\sum^{n}_{k=1}(X_{ik}-\bar{X}_{i})^2}}.
\end{equation*}
It is easy to see that
\begin{equation*}
\sum^n_{k=1}u_{ik}=0,\quad\sum^n_{k=1}u_{ik}^2=1,\quad\Rightarrow\quad E(u_{ik})=0, \quad E(u_{ik}^2)=\frac{1}{n}.
\end{equation*}
Furthermore,  for $k\neq k'$, we have
\begin{equation}\label{it2}
E(u_{ik}u_{ik'})=\frac{-1}{n(n-1)}.
\end{equation}
Denoting $u_{i}=(u_{i1},u_{i2},\cdots,u_{in})$, we have $r_{ij}=u_{i}u_{j}'$, and
\begin{align}\label{rij2}
E(r_{ij}^2|X_{i}) &=E(u_{i}u_{j}'u_{j}u_{i}'|X_{i})
=u_{i}E(u_{j}'u_{j})u_{i}'\notag\\
&=\frac{1}{n}\sum_{k=1}^{n}u_{ik}^2-\frac{1}{n(n-1)}\sum_{k_{1}\neq k_{2}}u_{ik_{1}}u_{ik_{2}}\notag\\
&=\frac{1}{n-1}.
\end{align}
By Chen and Shao \cite{cs2}, under condition $E(X_{11}^6)<\infty$, we derive the following relations for latge $n$:
\begin{align}
&E(u_{ik}^4)=\frac{K}{n^2}+O\Big(\frac{1}{n^3}\Big),\label{lz2}\\
& E(u_{ik_{1}}^3u_{ik_{2}})=-\frac{K}{n^2(n-1)}+O\Big(\frac{1}{n^4}\Big),\label{lz3}\\
& E(u_{ik}^2u_{ik'}^2)=-\frac{K}{n^2(n-1)}+\frac{1}{n(n-1)}+O\Big(\frac{1}{n^4}\Big).\label{lz4} \\ 
&E(u_{ik_{1}}^2u_{ik_{2}}u_{ik_{3}})=\frac{2K}{n^2(n-1)(n-2)}-\frac{1}{n(n-1)(n-2)}+O\Big(\frac{1}{n^5}\Big),\label{lz5}\\
&E(u_{ik_{1}}u_{ik_{2}}u_{ik_{3}}u_{ik_{4}})=\frac{3}{n(n-1)(n-2)(n-3)}-\frac{6K}{n^2(n-1)(n-2)(n-3)}+O\Big(\frac{1}{n^6}\Big),\label{lz6}\\
&E(r_{ij}^4)=\frac{3}{n^2}+O\Big(\frac{1}{n^3}\Big),\label{lz8}\\
&E(r_{ij}^4|X_{i})=\frac{3}{n^2}+O\Big(\frac{1}{n^3}\Big)+\Big(\frac{K-3}{n^2}+O\Big(\frac{1}{n^3}\Big)\Big)\sum_{k=1}^nu_{ik}^4.\label{lz7}
\end{align}
Here 
\begin{equation*}
K=\frac{E(X_{11}-\mu)^4}{\sigma^4}.
\end{equation*}

\begin{lemma} Under   $E(X_{11}^6)<\infty$, for large $n$,
	\begin{equation}\label{lemma1}
	E(r_{ij}^8)=O(\frac{1}{n^3}).
	\end{equation}
\end{lemma}

\begin{proof}
	
	Since $ r_{ij}^2\leq 1$, it suffices to show that for large $n$,
	\begin{equation*}
	E(r_{ij}^6)=O(\frac{1}{n^3}).
	\end{equation*}
	Indeed we first write $E(r_{ij}^6)$ as follows:
	\begin{align*}
	E(r_{ij}^6)=&E(r_{ij}^4r_{ij}^2)\\
	=&E\Big(\sum_{k=1}^nu_{jk}^4u_{ik}^4+4\sum_{k_{1}=1}^n\sum_{k_{2}=1\atop k_{2}\neq k_{1}}^nu_{jk_{1}}^3u_{jk_{2}}u_{ik_{1}}^3u_{ik_{2}}\\
	+&3\sum_{k_{1}=1}^n\sum_{k_{2}=1\atop k_{2}\neq k_{1}}^nu_{jk_{1}}^2u_{jk_{2}}^2u_{ik_{1}}^2u_{ik_{2}}^2+6\sum_{k_{1}=1}^n\sum_{k_{2}=1\atop k_{2}\neq k_{1}}^n\sum_{\substack{k_{3}=1\\k_{3}\neq k_{1}\\k_{3}\neq k_{2}}}^nu_{jk_{1}}^2u_{jk_{2}}u_{jk_{3}}u_{ik_{1}}^2u_{ik_{2}}u_{ik_{3}}\\
	+&\sum_{k_{1}=1}^n\sum_{k_{2}=1\atop k_{2}\neq k_{1}}^n\sum_{\substack{k_{3}=1\\k_{3}\neq k_{1}\\k_{3}\neq k_{2}}}^n\sum_{\substack{k_{4}=1\\k_{4}\neq k_{1}\\k_{4}\neq k_{2}\\k_{4}\neq k_{3}}}^nu_{jk_{1}}u_{jk_{2}}u_{jk_{3}}u_{jk_{4}}u_{ik_{1}}u_{ik_{2}}u_{ik_{3}}u_{ik_{4}}\Big)\\
	\times&\Big(\sum_{k=1}^nu_{ik}^2u_{jk}^2+\sum_{k_{1}\neq k_{2}}u_{ik_{1}}u_{ik_{2}}u_{jk_{1}}u_{jk_{2}}\Big).
	\end{align*}
	The next step is to derive the following relations, for large $n$:
	\begin{align*}
	& &E(u_{ik}^4u_{ik'}^2)=O\Big(\frac{1}{n^3}\Big)\label{bb}, ~~|Eu_{ik_{1}}^5u_{ik_{2}}|=O\Big(\frac{1}{n^3}\Big),\\
	&  &|Eu_{ik_{1}}^4u_{ik_{2}}u_{ik_{3}}|=O\Big(\frac{1}{n^4}\Big), ~~|Eu_{ik_{1}}^3u_{ik_{2}}^3|=O\Big(\frac{1}{n^3}\Big),\\
	& &|Eu_{ik_{1}}^3u_{ik_{2}}^2u_{ik_{3}}|=O\Big(\frac{1}{n^3}\Big),~~|Eu_{ik_{1}}^3u_{ik_{2}}u_{ik_{3}}u_{ik_{4}}|=O\Big(\frac{1}{n^4}\Big),\\
	& &Eu_{ik_{1}}^2u_{ik_{2}}^2u_{ik_{3}}^2=O\Big(\frac{1}{n^3}\Big),~~|Eu_{ik_{1}}^2u_{ik_{2}}^2u_{ik_{3}}u_{ik_{4}}|=O\Big(\frac{1}{n^4}\Big),\\
	& &|Eu_{ik_{1}}^2u_{ik_{2}}u_{ik_{3}}u_{ik_{4}}u_{ik_{5}}|=O\Big(\frac{1}{n^5}\Big),~~|Eu_{ik_{1}}u_{ik_{2}}u_{ik_{3}}u_{ik_{4}}u_{ik_{5}}u_{ik_{6}}|=O\Big(\frac{1}{n^6}\Big).
	\end{align*}
	
	As an example, we just calculate the first three items. The other items can be proved in a similar way by (\ref{lz2})$\sim$(\ref{lz6}).
	
	For $E(u_{ik}^4u_{ik'}^2)=O\Big(\frac{1}{n^3}\Big)$, we have by (\ref{lz2}):
	\begin{align*}
	E(u_{ik}^4u_{ik'}^2)=&\frac{1}{n-1}E(u_{ik}^4(\sum_{k'=1\atop k'\neq k}^nu_{ik'}^2)\\
	=&\frac{1}{n-1}[E(u_{ik}^4(\sum_{k'=1}^nu_{ik'}^2))-E(u_{ik}^6)]\\
	\leq&\frac{2}{n-1}E(u_{ik}^4)\\
	=&\frac{2K}{n^2(n-1)}+O(\frac{1}{n^4}).
	\end{align*}
	For $|Eu_{ik_{1}}^5u_{ik_{2}}|=O\Big(\frac{1}{n^3}\Big)$, we have by (\ref{lz2}):
	\begin{align*}
	|Eu_{ik_{1}}^5u_{ik_{2}}|=&\Big|\frac{1}{(n-1)}E\Big(u_{ik_{1}}^5\Big(\sum_{k\neq k_{1}}u_{ik}\Big)\Big)\Big|\\
	=&\frac{1}{(n-1)}\Big|E\Big(u_{ik_{1}}^5\Big(\sum_{k=1}^nu_{ik}-u_{ik_{1}}\Big)\Big|\\
	=&\frac{1}{(n-1)}E(u_{ik}^6)\\
	=&O\Big(\frac{1}{n^3}\Big)	.
	\end{align*}
	For $|Eu_{ik_{1}}^4u_{ik_{2}}u_{ik_{3}}|=O\Big(\frac{1}{n^4}\Big)$,  by (\ref{it2}) , (\ref{lz2}) and above two conclusions about $|Eu_{ik_{1}}^5u_{ik_{2}}|  $ and $ E(u_{ik}^4u_{ik'}^2) $, we have:
	\begin{align*}
	|Eu_{ik_{1}}^4u_{ik_{2}}u_{ik_{3}}|=&\Big|\frac{1}{(n-1)(n-2)}E\Big[u_{ik_{1}}^4\Big(\sum_{k\neq k'}u_{ik}u_{ik'}\Big)-2\sum_{k_{2}\neq k_{1}}u_{ik_{1}}^5u_{ik_{2}}\Big]\Big|\\
	\leq&\frac{1}{(n-1)(n-2)}\Big(\Big|Eu_{ik_1}^4\Big|+2(n-1)\Big|Eu_{ik_{1}}^5u_{ik_2}\Big|\Big)\\
	=&O\Big(\frac{1}{n^4}\Big).
	\end{align*}
	In a similar way, we can obtain all other relations. Then the lemma can be proved.
	
\end{proof}

\begin{lemma} \label{l2} Under  the condition $E(X_{11}^6)<\infty$, for large $n$,
	
	\begin{equation}\label{l21}
	\sqrt{E\Big|1-\frac{1}{2\lambda}E(\Delta^2|\textbf{X})\Big|^2}=O\Big(\frac{1}{p^{1/2}}\Big).
	\end{equation}
	
\end{lemma}

\begin{proof}
	
	Recall that $ \lambda=\frac{2}{p} $. Chen and Shao \cite{cs2} obtained the following relation:	
	\begin{align*}
	\Big|1-\frac{1}{2\lambda}E(\Delta^2|\textbf{X})\Big|\leq &\frac{p\cdot c_{n,p}^2}{4}(J_{1}+J_{2})+O\Big(\frac{1}{n}\Big),
	\end{align*}
	where
	\begin{align*}
	J_{1}=&\Big|\frac{1}{p}\sum_{i=1}^p\Big(\sum_{j=1\atop j\neq i}\Big(r_{ij}^2-\frac{1}{n-1}\Big)\Big)^2-\frac{2(p-1)}{n^2}\Big|,\\
	J_{2}=&\Big|\frac{1}{p}\sum_{i=1}^pE\Big(\Big(\sum_{j=1\atop j\neq i}\Big(r_{i^*j}^2-\frac{1}{n-1}\Big)\Big)^2|\textbf{X}\Big)-\frac{2(p-1)}{n^2}\Big|.
	\end{align*}
	By Jensen's inequality of conditional expectation,  we only need to estimate $EJ_{1}$.
	It is easy to check that
	\begin{align*}
	E\Big(\frac{1}{p}\sum_{i=1}^p\Big(\sum_{j=1\atop j\neq i}^p\Big(r_{ij}^2-\frac{1}{n-1}\Big)\Big)^2)=&E\Big(\Big(\sum_{j=1\atop j\neq i}^p\Big(r_{ij}^2-\frac{1}{n-1}\Big)\Big)^2\Big)\\
	=&(p-1)\Big(E(r_{ij}^4)-\frac{1}{(n-1)^2}\Big)=(p-1)\frac{2}{n^2}+O\Big(\frac{1}{n^2}\Big).
	\end{align*}
	Hence we have
	\begin{align}
\notag	EJ_{1}^2=&E\Big\{\Big[\frac{1}{p}\sum_{i=1}^p\Big(\sum_{j=1\atop j\neq i}^p\Big(r_{ij}^2-\frac{1}{n-1}\Big)\Big)^2-\frac{2(p-1)}{n^2}\Big]^2\Big\}\\
\notag	=&E\Big\{\Big[\frac{1}{p}\sum_{i=1}^p\Big(\sum_{j=1\atop j\neq i}^p\Big(r_{ij}^2-\frac{1}{n-1}\Big)\Big)^2\Big]^2\Big\}-\frac{4(p-1)^2}{n^4}+O\Big(\frac{1}{n^3}\Big)\\
\notag	=&\frac{1}{p^2}E\Big\{\sum_{i=1}^p\Big[\sum_{j=1\atop j\neq i}^p\Big(r_{ij}^2-\frac{1}{n-1}\Big)\Big]^4+\sum_{i\neq i'}^p\Big[\sum_{j=1\atop j\neq i}^p\Big(r_{ij}^2-\frac{1}{n-1}\Big)\Big]^2\Big[\sum_{j=1\atop j\neq i'}^p\Big(r_{i'j}^2-\frac{1}{n-1}\Big)\Big]^2\Big\}\\
	& -\frac{4(p-1)^2}{n^4}+O\Big(\frac{1}{n^3}\Big).\label{l22}
	\end{align}
	
	We use the last expression for $EJ_{1}^2$ and estimate each term in order to show that $EJ_{1}^2=O\Big(\frac{1}{n^3}\Big)$. \\
	For the first item on the right side of the last equality of (\ref{l22}), by (\ref{rij2}), (\ref{lz8}), (\ref{lz7}) and (\ref{lemma1}), we have:
	
	\begin{align}
	 \notag&~\quad E\Big(\sum_{j=1\atop j\neq i}^p\Big(r_{ij}^2-\frac{1}{n-1}\Big)\Big)^4\\
	\notag&=
	E\Big\{\Big(\sum_{j=1\atop j\neq i}^p\Big(r_{ij}^2-\frac{1}{n-1}\Big)^4\Big)+4\sum_{j_{1}\neq j_{2}\atop j_{1},j_{2} \neq i}E\Big(\Big(r_{ij_{1}}^2-\frac{1}{n-1}\Big)^3\Big|X_{i}\Big)E\Big(\Big(r_{ij_{2}}^2-\frac{1}{n-1}\Big)\Big|X_{i}\Big)\\
	\notag&+3\sum_{j_{1}\neq j_{2}\atop j_{1},j_{2}\neq i}E\Big(\Big(r_{ij_{1}}^2-\frac{1}{n-1}\Big)^2\Big|X_{i}\Big)E\Big(\Big(r_{ij_{2}}^2-\frac{1}{n-1}\Big)^2\Big|X_{i}\Big)\\
	\notag&+6\sum_{j_{1}\neq i}^p\sum_{ \substack{j_{2}\neq j_{1}\\ j_{2}\neq i}}^p\sum_{\substack{j_{3}\neq j_{1}\\ j_{3}\neq j_{2}\\j_{3}\neq i}}^pE\Big(\Big(r_{ij_{1}}^2-\frac{1}{n-1}\Big)^2\Big|X_{i}\Big)E\Big(\Big(r_{ij_{2}}^2-\frac{1}{n-1}\Big)\Big|X_{i}\Big)E\Big(\Big(r_{ij_{3}}^2-\frac{1}{n-1}\Big)\Big|X_{i}\Big)\\
	\notag&+\sum_{j_{1}\neq i}^p\sum_{\substack{j_{2}\neq j_{1}\\ j_{2}\neq i}}^p\sum_{\substack{j_{3}\neq j_{1}\\j_{3}\neq j_{2}\\j_{3}\neq i}}^p\sum_{\substack{j_{4}\neq j_{1}\\j_{4}\neq j_{2}\\j_{4}\neq j_{3}\\j_{4}\neq i}}^p\Big[E\Big(\Big(r_{ij_{1}}^2-\frac{1}{n-1}\Big)\Big|X_{i}\Big)E\Big(\Big(r_{ij_{2}}^2-\frac{1}{n-1}\Big)\Big|X_{i}\Big)\\
	\notag&\times E\Big(\Big(r_{ij_{3}}^2-\frac{1}{n-1}\Big)\Big|X_{i}\Big)E\Big(\Big(r_{ij_{4}}^2-\frac{1}{n-1}\Big)\Big|X_{i}\Big)\Big]\Big\}\\
	\notag&=	E\Big\{\Big(\sum_{j=1\atop j\neq i}^p\Big(r_{ij}^2-\frac{1}{n-1}\Big)^4\Big)+3\sum_{j_{1}\neq j_{2}\atop j_{1},j_{2} \neq i}E\Big(\Big(r_{ij_{1}}^2-\frac{1}{n-1}\Big)^2\Big|X_{i}\Big)E\Big(\Big(r_{ij_{2}}^2-\frac{1}{n-1}\Big)^2\Big|X_{i}\Big)\Big\}\\
	&=O\Big(\frac{1}{n^2}\Big)\label{l23}.
	\end{align}
	
	For the second item on the right side of the last equality of (\ref{l22}), we have 
	\begin{align*}
	&\frac{1}{p^2}\sum_{i\neq i'}^pE\Big(\sum_{j=1\atop j \neq i}^p\Big(r_{ij}^2-\frac{1}{n-1}\Big)\Big)^2\Big(\sum_{j=1\atop j\neq i'}^p\Big(r_{i'j}^2-\frac{1}{n-1}\Big)\Big)^2\\
	=&\frac{1}{p^2}\sum_{i\neq i'}^pE\Big(\sum_{j=1\atop j\neq i,i'}^p\Big(r_{ij}^2-\frac{1}{n-1}\Big)+r_{ii'}^2-\frac{1}{n-1}\Big)^2\Big(\sum_{j=1\atop j\neq i,i'}^p\Big(r_{i'j}^2-\frac{1}{n-1}\Big)+r_{ii'}^2-\frac{1}{n-1}\Big)^2.
	\end{align*}
	Also, we can see
	\begin{align*}
	&E\Big(\sum_{j=1\atop j\neq i,i'}^p\Big(r_{ij}^2-\frac{1}{n-1}\Big)+r_{ii'}^2-\frac{1}{n-1}\Big)^2\Big(\sum_{j=1\atop j\neq i,i'}^p\Big(r_{i'j}^2-\frac{1}{n-1}\Big)+r_{ii'}^2-\frac{1}{n-1}\Big)^2\\
	=&E\Big[\Big(\sum_{j=1\atop j\neq i,i'}^p\Big(r_{ij}^2-\frac{1}{n-1}\Big)\Big)^2+\Big(r_{ii'}^2-\frac{1}{n-1}\Big)^2+2\Big(\sum_{j=1\atop j\neq i,i'}^p\Big(r_{ij}^2-\frac{1}{n-1}\Big)\Big)\Big(r_{ii'}^2-\frac{1}{n-1}\Big)\Big]\\ &\times \Big[\Big(\sum_{j=1\atop j\neq i,i'}^p\Big(r_{i'j}^2-\frac{1}{n-1}\Big)\Big)^2+\Big(r_{ii'}^2-\frac{1}{n-1}\Big)^2+2\Big(\sum_{j=1\atop j\neq i,i'}^p\Big(r_{i'j}^2-\frac{1}{n-1}\Big)\Big)\Big(r_{ii'}^2-\frac{1}{n-1}\Big)\Big].
	\end{align*}
	To estimate the above item, we need to estimate 
	\begin{align*}
&E\Big(r_{ij}^2- \frac{1}{n-1}\Big)\Big(r_{ij'}^2- \frac{1}{n-1}\Big)\Big(\sum_{j=1\atop j\neq i,i'}^nr_{i'j}^2-\frac{1}{n-1}\Big)^2,\quad E\Big(\sum_{j=1\atop j\neq i,i'}^pr_{ij}^2-\frac{1}{n-1}\Big)^2\Big(\sum_{j=1\atop j\neq i,i'}^pr_{i'j}^2-\frac{1}{n-1}\Big)^2,\\
&E\Big(\sum_{j=1\atop j\neq i,i'}\Big(r_{ij}^2-\frac{1}{n-1}\Big)\Big)^2\Big(r_{ii'}^2-\frac{1}{n-1}\Big)^2,\quad E\Big(\sum_{j=1\atop j\neq i,i'}\Big(r_{ij}^2-\frac{1}{n-1}\Big)\Big)^2\Big(\sum_{j=1\atop j\neq i,i'}r_{i'j}^2-\frac{1}{n-1}\Big)\Big(r_{ii'}^2-\frac{1}{n-1}\Big),\\
&E\Big(r_{ii'}^2-\frac{1}{n-1}\Big)^3\Big(\sum_{j=1\atop j\neq i,i'}\Big(r_{i'j}^2-\frac{1}{n-1}\Big)\Big), \quad E\Big(r_{ii'}^2-\frac{1}{n-1}\Big)^2\Big(\sum_{j=1\atop j\neq i,i'}\Big(r_{i'j}^2-\frac{1}{n-1}\Big)\Big)\Big(\sum_{j=1\atop j\neq i,i'}\Big(r_{ij}^2-\frac{1}{n-1}\Big)\Big).
	\end{align*}
	After some simplification,  we can see
	\begin{align*}
	&E\Big(r_{ij}^2- \frac{1}{n-1}\Big)\Big(r_{ij'}^2- \frac{1}{n-1}\Big)\Big(\sum_{j=1\atop j\neq i,i'}^nr_{i'j}^2-\frac{1}{n-1}\Big)^2\\
	=&E\Big\{E\Big[\sum_{j''=1\atop j''\neq i,i',j,j'}\Big(r_{ij}^2- \frac{1}{n-1}\Big)\Big(r_{ij'}^2- \frac{1}{n-1}\Big)\Big(r_{i'j''}^2-\frac{1}{n-1}\Big)^2\Big|(X_{i},X_{i'})\Big]\\
	&+E\Big[\Big(r_{ij}^2- \frac{1}{n-1}\Big)\Big(r_{ij'}^2- \frac{1}{n-1}\Big)\Big(\Big(r_{i'j'}^2-\frac{1}{n-1}\Big)^2+\Big(r_{i'j}^2-\frac{1}{n-1}\Big)^2\Big)\Big|(X_{i},X_{i'})\Big]\\
	&+E\Big[2\Big(r_{ij}^2- \frac{1}{n-1}\Big)\Big(r_{ij'}^2- \frac{1}{n-1}\Big)\Big(\sum_{j''=1\atop j''\neq i,i',j,j'}\Big(r_{i'j}^2- \frac{1}{n-1}\Big)\Big(r_{i'j''}^2- \frac{1}{n-1}\Big)\Big)\Big|(X_{i},X_{i'})\Big]\\
	&+E\Big[2\Big(r_{ij}^2- \frac{1}{n-1}\Big)\Big(r_{ij'}^2- \frac{1}{n-1}\Big)\Big(\sum_{j''=1\atop j''\neq i,i',j,j'}\Big(r_{i'j'}^2-\frac{1}{n-1}\Big)\Big(r_{i'j''}^2-\frac{1}{n-1}\Big)\Big)\Big|(X_{i},X_{i'})\Big]\\
	&+E\Big[\Big(r_{ij}^2- \frac{1}{n-1}\Big)\Big(r_{ij'}^2- \frac{1}{n-1}\Big)\Big(\sum_{j_{1}\neq j_{2}\atop j_{1},j_{2}\neq j,j',i,i'}\Big(r_{i'j_{1}}^2- \frac{1}{n-1}\Big)\Big(r_{i'j_{2}}^2- \frac{1}{n-1}\Big)\Big|(X_{i},X_{i'})\Big)\Big]\\
	&+2\Big(r_{ij}^2- \frac{1}{n-1}\Big)\Big(r_{ij'}^2- \frac{1}{n-1}\Big)\Big(r_{i'j}^2-\frac{1}{n-1}\Big)\Big(r_{i'j'}^2-\frac{1}{n-1}\Big)\Big\}\\
	=&2E\Big(r_{ij}^2- \frac{1}{n-1}\Big)\Big(r_{ij'}^2- \frac{1}{n-1}\Big)\Big(r_{i'j}^2-\frac{1}{n-1}\Big)\Big(r_{i'j'}^2-\frac{1}{n-1}\Big),
	\end{align*}
	and
	\begin{align*}
	&E\Big(r_{ij}^2- \frac{1}{n-1}\Big)\Big(r_{ij'}^2- \frac{1}{n-1}\Big)\Big(r_{i'j}^2-\frac{1}{n-1}\Big)\Big(r_{i'j'}^2-\frac{1}{n-1}\Big)\\
	=&E\Big\{E\Big[\Big(r_{ij}^2- \frac{1}{n-1}\Big)\Big(r_{i'j}^2-\frac{1}{n-1}\Big)\Big|(X_{i},X_{i'})\Big]E\Big[\Big(r_{ij'}^2- \frac{1}{n-1}\Big)\Big(r_{i'j'}^2-\frac{1}{n-1}\Big)\Big|(X_{i},X_{i'})\Big]\\
	=&E\{E[r_{ij}^2r_{i'j}^2|(X_{i},X_{i'})]\}^2-\frac{1}{(n-1)^4}.
	\end{align*}
	Thus, by the relations we derive in the proof of Lemma \ref{l2}, we obtain
	\begin{align*}
	&E(r_{ij}^2r_{i'j}^2|(X_{i},X_{i'}))=E\Big(\Big(\sum_{k=1}^nu_{ik}u_{jk}\Big)^2\Big(\sum_{k=1}^nu_{i'k}u_{jk}\Big)^2\Big|(X_{i},X_{i'})\Big)\\
	=&E\Big\{\Big[\sum_{k=1}^n\Big(u_{ik}u_{jk}\Big)^2+\sum_{k_{1}\neq k_{2}}u_{ik_{1}}u_{ik_{2}}u_{jk_{1}}u_{jk_{2}}\Big]\\
	&\times\Big[\sum_{k=1}^n\Big(u_{i'k}u_{jk}\Big)^2+\sum_{k_{1}\neq k_{2}}u_{i'k_{1}}u_{i'k_{2}}u_{jk_{1}}u_{jk_{2}}\Big]\Big|(X_{i},X_{i'})\Big\}\\
	=&\sum_{k=1}^nu_{ik}^2u_{i'k}^2E(u_{jk}^4)+\sum_{k=1}^n\sum_{k'=1\atop k'\neq k}^nu_{ik}^2u_{i'k}^2E(u_{jk}^2u_{jk'}^2)\\
	&+2\sum_{k=1}^n\sum_{k_{1}\neq k}^nu_{ik}^2u_{i'k}u_{i'k_{1}}E(u_{jk}^3u_{jk_{1}})+\sum_{k=1}^n\sum_{k_{1}\neq k}^n\sum_{k_2\neq k_1\atop k_2\neq k}^nu_{ik}^2u_{i'k_{1}}u_{i'k_{2}}E(u_{jk}^2u_{jk_{1}}u_{jk_{2}})\\
	&+2\sum_{k=1}^n\sum_{k_{1}\neq k}^nu_{i'k}^2u_{ik}u_{ik_{1}}E(u_{jk}^3u_{jk_{1}})+\sum_{k=1}^n\sum_{k_{1}\neq k}^n\sum_{k_2\neq k_1\atop k_2\neq k}^nu_{i'k}^2u_{ik_{1}}u_{ik_{2}}E(u_{jk}^2u_{jk_{1}}u_{jk_{2}})\\
	&+2\sum_{k_{1}\neq k_{2}}^nu_{ik_{1}}u_{ik_{2}}u_{i'k_{1}}u_{i'k_{2}}E(u_{jk}^2u_{jk'}^2)+4\sum_{k_{1}\neq k_{2}}^n\sum_{k_{3}\neq k_{1}\atop k_{3} \neq k_{2}}^nu_{ik_{1}}u_{ik_{2}}u_{i'k_{1}}u_{i'k_{3}}E(u_{jk}^2u_{jk_{1}}u_{jk_{2}})\\
	&+\sum_{k_{1}\neq k_{2}}^n\sum_{k_{3}\neq k_{1}\atop k_{3}\neq k_{2}}^n\sum_{\substack{k_{4}\neq k_{1}\\k_{4}\neq k_{2}\\k_{4}\neq k_{3}}}^nu_{ik_{1}}u_{ik_{2}}u_{i'k_{3}}u_{i'k_{4}}E(u_{jk_{1}}u_{jk_{2}}u_{jk_{3}}u_{jk_{4}})\\
	=&\sum_{k=1}^nu_{ik}^2u_{i'k}^2\Big(O\Big(\frac{1}{n^2}\Big)\Big)
	+\sum_{k_{1}\neq k_{2}}^nu_{ik_{1}}u_{ik_{2}}u_{i'k_{1}}u_{i'k_{2}}\Big(O\Big(\frac{1}{n^2}\Big)\Big)\\
	&+\sum_{k_{1}\neq k_{2}}^nu_{ik_{1}}u_{ik_{2}}u_{i'k_{1}}^2\Big(O\Big(\frac{1}{n^3}\Big)\Big)+\frac{1}{n(n-1)}+O\Big(\frac{1}{n^3}\Big).
	\end{align*}
	We use the above findings to derive that
	\begin{align}
	\notag&E\{E[r_{ij}^2r_{i'j}^2|(X_{i},X_{i'})]\}^2-\frac{1}{(n-1)^4}\\
	\notag=&E\Big[\sum_{k=1}^nu_{ik}^2u_{i'k}^2\Big(O\Big(\frac{1}{n^2}\Big)\Big)
	+\sum_{k_{1}\neq k_{2}}^nu_{ik_{1}}u_{ik_{2}}u_{i'k_{1}}u_{i'k_{2}}\Big(O\Big(\frac{1}{n^2}\Big)\Big)\\
	\notag&+\sum_{k_{1}\neq k_{2}}^nu_{ik_{1}}u_{ik_{2}}u_{i'k_{1}}^2\Big(O\Big(\frac{1}{n^3}\Big)\Big)+\frac{1}{n(n-1)}+O\Big(\frac{1}{n^3}\Big)\Big]^2-\frac{1}{(n-1)^4}\\
	=&O\Big(\frac{1}{n^5}\Big).\label{l24}
	\end{align}
	The relation (\ref{l24}) also shows that
	\begin{align*}
	&E\Big(\sum_{j=1\atop j\neq i,i'}^p\Big(r_{ij}^2-\frac{1}{n-1}\Big)\Big)^2\Big(\sum_{j=1\atop j\neq i,i'}^p\Big(r_{i'j}^2-\frac{1}{n-1}\Big)\Big)^2\\
	=&E\Big(\sum_{j=1\atop j\neq i,i'}^p\Big(r_{ij}^2-\frac{1}{n-1}\Big)^2\Big)\Big(\sum_{j=1\atop j\neq i,i'}^p\Big(r_{i'j}^2-\frac{1}{n-1}\Big)^2\Big)+O\Big(\frac{1}{n^3}\Big).
	\end{align*}
	In the following, we define  $\mathcal{F}_{i}=\sigma(X_j,j\neq i), \mathcal{F}_{i,i'}=\sigma(X_j,j\neq i, i')$. Then we have by (\ref{lz7})
	\begin{align*}
	&\quad E\Big(\sum_{j=1\atop j\neq i,i'}^pr_{ij}^2-\frac{1}{n-1}\Big)^2\Big(\sum_{j=1\atop j\neq i,i'}^pr_{i'j}^2-\frac{1}{n-1}\Big)^2\\
	&=E\Big\{E\Big[\sum_{j=1\atop j\neq i,i'}^p\Big(r_{ij}^2-\frac{1}{n-1}\Big)^2\Big|\mathcal{F}_{i,i'}\Big]E\Big[\sum_{j=1\atop j\neq i,i'}^p\Big(r_{i'j}^2-\frac{1}{n-1}\Big)^2\Big|\mathcal{F}_{i,i'}\Big]\Big\}+O\Big(\frac{1}{n^3}\Big)\\
	&=E\Big\{\sum_{j=1\atop j\neq i,i'}^p\Big[E\Big(r_{ij}^4\Big|\mathcal{F}_{i,i'}\Big)-\frac{1}{(n-1)^2}\Big]\Big\}\Big\{\sum_{j=1\atop j\neq i,i'}^p\Big[E\Big(r_{i'j}^4\Big|\mathcal{F}_{i,i'}\Big)-\frac{1}{(n-1)^2}\Big]\Big\}\\
	&=E\Big\{\sum_{j=1\atop j\neq i,i'}^p\frac{2}{n^2}+O\Big(\frac{1}{n^3}\Big)+\Big[\frac{K-3}{n^2}+O\Big(\frac{1}{n^3}\Big)\Big]\sum_{k=1}^nu_{jk}^4\Big\}\\
	&\times\Big\{\sum_{j=1\atop j\neq i,i'}^p\frac{2}{n^2}+O\Big(\frac{1}{n^3}\Big)+\Big[\frac{K-3}{n^2}+O\Big(\frac{1}{n^3}\Big)\Big]\sum_{k=1}^nu_{jk}^4\Big\}\\
	&=E\Big\{\frac{2p}{n^2}+O\Big(\frac{1}{n^2}\Big)+\Big[\frac{K-3}{n^2}+O\Big(\frac{1}{n^3}\Big)\Big]\sum_{j=1\atop j\neq i,i'}^p\sum_{k=1}^nu_{jk}^4\Big\}\\
	&\times\Big\{\frac{2p}{n^2}+O\Big(\frac{1}{n^2}\Big)+\Big[\frac{K-3}{n^2}+O\Big(\frac{1}{n^3}\Big)\Big]\sum_{j=1\atop j\neq i,i'}^p\sum_{k=1}^nu_{jk}^4\Big\}\\
	&=\frac{4p^2}{n^4}+O\Big(\frac{1}{n^3}\Big)+E\Big\{\Big[\frac{K-3}{n^2}+O\Big(\frac{1}{n^3}\Big)\Big]^2\Big(\sum_{j=1\atop j\neq i,i'}^p\sum_{k=1}^nu_{jk}^4\Big)^2\Big\}\\
	&=\frac{4p^2}{n^4}+O\Big(\frac{1}{n^3}\Big)+\Big[\frac{(K-3)^2}{n^4}+O\Big(\frac{1}{n^5}\Big)\Big]E\Big[\sum_{j=1\atop j\neq i,i'}^p\Big(\sum_{k=1}^nu_{jk}^4\Big)^2+\sum_{j\neq j'\atop j,j'\neq i,i'}\Big(\sum_{k=1}^nu_{jk}^4\Big)\Big(\sum_{k=1}^nu_{j'k}^4\Big)\Big]\\
	&\leq\frac{4p^2}{n^4}+O\Big(\frac{1}{n^3}\Big)+\Big[\frac{(K-3)^2}{n^4}+O\Big(\frac{1}{n^5}\Big)\Big]\sum_{j\neq j'\atop j,j'\neq i,i'}E\Big(\sum_{k=1}^nu_{jk}^4\Big)E\Big(\sum_{k=1}^nu_{j'k}^4\Big)\\
	&=\frac{4p^2}{n^4}+O\Big(\frac{1}{n^3}\Big).
	\end{align*}
	Here we use the fact $\sum_{k=1}^nu_{jk}^4\leq \sum_{k=1}^nu_{jk}^2=1$.		In view of the above we can conclude that,
	\begin{equation}
	\label{n1}
	E\Big(\sum_{j=1\atop j\neq i,i'}^pr_{ij}^2-\frac{1}{n-1}\Big)^2\Big(\sum_{j=1\atop j\neq i,i'}^pr_{i'j}^2-\frac{1}{n-1}\Big)^2=\frac{4p^2}{n^4}+O\Big(\frac{1}{n^3}\Big).\end{equation}
	Moreover, by (\ref{lz7}), we have
	\begin{align*}
	&E\Big(\sum_{j=1\atop j\neq i,i'}\Big(r_{ij}^2-\frac{1}{n-1}\Big)\Big)^2\Big(r_{ii'}^2-\frac{1}{n-1}\Big)^2\\
	=&E\Big\{E\Big[\Big(\sum_{j=1\atop j\neq i,i'}r_{ij}^2-\frac{1}{n-1}\Big)^2\Big|X_{i}\Big]E\Big[\Big(r_{ii'}^2-\frac{1}{n-1}\Big)^2\Big|X_{i}\Big]\Big\}\\
	=&E\Big\{\frac{2p}{n^2}+O\Big(\frac{1}{n^2}\Big)+(p-2)\Big[\frac{K-3}{n^2}+O\Big(\frac{1}{n^3}\Big)
	\sum_{k=1}^nu_{ik}^4\Big]\Big\}\\
	\times&\Big\{\frac{2}{n^2}+O\Big(\frac{1}{n^3}\Big)+\Big[\frac{K-3}{n^2}+O\Big(\frac{1}{n^3}\Big)\Big]\sum_{k=1}^nu_{ik}^4\Big\}\\
	\leq&E\Big[\frac{2p}{n^2}+\frac{K-3}{n}+O\Big(\frac{1}{n^2}\Big)\Big]\Big[\frac{2}{n^2}+\frac{K-3}{n^2}+O\Big(\frac{1}{n^3}\Big)\Big]\\
	=&O\Big(\frac{1}{n^3}\Big)
	\end{align*}
	which is obtained by applying the inequality $\Big(\sum_{k=1}^nu_{jk}^4\Big)^2\leq 1$ .
	
	Thus, we have:
	\begin{equation}
	\label{n2}
	E\Big(\sum_{j=1\atop j\neq i,i'}\Big(r_{ij}^2-\frac{1}{n-1}\Big)\Big)^2\Big(r_{ii'}^2-\frac{1}{n-1}\Big)^2=O\Big(\frac{1}{n^3}\Big).\end{equation}
	
	Finally, we can obtain the results below by Cauchy's inequality,
	\begin{align}
	\notag&E\Big(\sum_{j=1\atop j\neq i,i'}\Big(r_{ij}^2-\frac{1}{n-1}\Big)\Big)^2\Big(\sum_{j=1\atop j\neq i,i'}\Big(r_{i'j}^2-\frac{1}{n-1}\Big)\Big)\Big(r_{ii'}^2-\frac{1}{n-1}\Big)\\
	\notag=&E\Big[\sum_{j=1\atop j\neq i,i'}\Big(r_{ij}^2-\frac{1}{n-1}\Big)^2\Big]\Big(\sum_{j=1\atop j\neq i,i'}\Big(r_{i'j}^2-\frac{1}{n-1}\Big)\Big)\Big(r_{ii'}^2-\frac{1}{n-1}\Big)\\
	\notag=&\sum_{j=1\atop j\neq i,i'}^pE\Big(r_{ij}^2-\frac{1}{n-1}\Big)^2\Big(r_{i'j}^2-\frac{1}{n-1}\Big)\Big(r_{ii'}^2-\frac{1}{n-1}\Big)\\
	\notag\leq&\Big(p-2\Big)\sqrt{E\Big(r_{ij}^2-\frac{1}{n-1}\Big)^2\Big(r_{i'j}^2-\frac{1}{n-1}\Big)^2}\sqrt{E\Big(r_{ij}^2-\frac{1}{n-1}\Big)^2\Big(r_{ii'}^2-\frac{1}{n-1}\Big)^2}\\
	=&O\Big(\frac{1}{n^3}\Big).\label{l25}
	\end{align}
	Also we have
	\begin{align}
	\notag&E\Big(r_{ii'}^2-\frac{1}{n-1}\Big)^3\Big(\sum_{j=1\atop j\neq i,i'}r_{i'j}^2-\frac{1}{n-1}\Big)\\
	\notag\leq&\sqrt{E\Big(r_{ii'}^2-\frac{1}{n-1}\Big)^2\Big(\sum_{j=1\atop j\neq i,i'}r_{i'j}^2-\frac{1}{n-1}\Big)^2}\sqrt{E\Big(r_{ii'}^2-\frac{1}{n-1}\Big)^4}\\
	=&O\Big(\frac{1}{n^3}\Big)\label{l26}
	\end{align}
	and
	\begin{align}
	\notag&E\Big(r_{ii'}^2-\frac{1}{n-1}\Big)^2\Big(\sum_{j=1\atop j\neq i,i'}\Big(r_{i'j}^2-\frac{1}{n-1}\Big)\Big)\Big(\sum_{j=1\atop j\neq i,i'}r_{ij}^2-\frac{1}{n-1}\Big)\\
	\notag\leq&\sqrt{E\Big(r_{ii'}^2-\frac{1}{n-1}\Big)^2\Big(\sum_{j=1\atop j\neq i,i'}\Big(r_{i'j}^2-\frac{1}{n-1}\Big)\Big)^2}\sqrt{E\Big(r_{ii'}^2-\frac{1}{n-1}\Big)^2\Big(\sum_{j=1\atop j\neq i,i'}\Big(r_{ij}^2-\frac{1}{n-1}\Big)\Big)^2}\\
	=&O\Big(\frac{1}{n^3}\Big).\label{l27}
	\end{align}
	Combining (\ref{l24}) to (\ref{l27}), we obtain
	\begin{equation}\label{l28}
    E\Big(\sum_{j=1\atop j\neq i}^p\Big(r_{ij}^2-\frac{1}{n-1}\Big)\Big)^2\Big(\sum_{j=1\atop j\neq i'}^p\Big(r_{i'j}^2-\frac{1}{n-1}\Big)\Big)^2\leq\frac{4p^2}{n^4}+O\Big(\frac{1}{n^3}\Big).
	\end{equation}
	By (\ref{l23}) and (\ref{l28}), recalling that $ p=O\Big(\frac{1}{n^3}\Big) $, we can easily get
	\begin{align}
     \notag EJ_1^2&\leq\frac{1}{p}O\Big(\frac{1}{n^2}\Big)+\frac{p(p-1)}{p^2}\cdot\frac{4p^2}{n^4}-\frac{4(p-1)^2}{n^4}+ O\Big(\frac{1}{n^3}\Big)\\
     &\leq O\Big(\frac{1}{n^3}\Big).
   \end{align}The proof is complete.
	
\end{proof}

\begin{lemma} Under  the condition $E(X_{11}^6)<\infty$,
	
	\begin{equation}\label{l3}
	\frac{1}{\lambda}\sqrt{E\Big|E\Big(\Delta|\Delta| \,\big|\textbf{X}\Big)\Big|^2  }=O\Big(\frac{1}{p^{1/2}}\Big).
	\end{equation}
\end{lemma}

\begin{proof}

	Define
	\begin{gather*}
	k_{i}=E\Big(\sum_{j=1\atop j\neq i}^p\Big(r_{ij}^2-r_{i^*j}^2\Big)\Big|\sum_{j=1\atop j\neq i}^p\Big(r_{ij}^2-r_{i^*j}^2\Big)\Big|\Big|\textbf{X}\Big),\\
	k_{i}^{(i')}=E\Big(\sum_{j=1\atop j\neq i,i'}^p\Big(r_{ij}^2-r_{i^*j}^2\Big)\Big|\sum_{j=1\atop j\neq i,i'}^p\Big(r_{ij}^2-r_{i^*j}^2\Big)\Big|\Big|\textbf{X}\Big),\\
	\end{gather*}
    and recall $	{\cal F}_{i}=\sigma(X_j,j\neq i),  {\cal F}_{i,i'}=\sigma(X_j,j\neq i,i').$
	Then
	\begin{align*}
	E\Big(\Delta|\Delta| \,\big|\textbf{X}\Big)=\frac{c_{n,p}^2}{p}E\sum_{i=1}^nk_{i}.
	\end{align*}
	By the symmetry property of $k_{i}$, we have
	\begin{align*}
		&E\Big(E\Big(k_{i}\Big|{\cal F}_{i}\Big)\Big)=0,\\
	&E\Big(k_{i}^{(i')}\Big)=E\Big(E(k_{i}^{(i')}|{\cal F}_{i})\Big)=0.
	\end{align*}
	Given the field ${\cal F}_{i,i'}  $, $ k_{i}^{(i')} $ is conditionally independent of $ k_{i'}^{(i)} $. Also, $ k_{i}^{(i')} $ is conditionally independent of $ k_{i'} $ given $ {\cal F}_{i} $. Thus we have
	\begin{align*}
	&Cov\Big(k_{i'}^{(i)},k_{i}^{(i')}\Big)=E\Big(E(k_{i'}^{(i)}|{\cal F}_{i,i'})\cdot E(k_{i}^{(i')}|{\cal F}_{i,i'})\Big)=0,\\
	&Cov\Big(k_{i}^{(i')},k_{i'}\Big)=E\Big(k_{i}^{(i')}\cdot E(k_{i'}|{\cal F}_{i})\Big)=0.
	\end{align*}
	Recall that $ \lambda=\frac{2}{p} $. It suffices to estimate $ Var\Big( E(\Delta|\Delta| \,\big|\textbf{X})\Big) $.
	\begin{align}\label{l31}
	Var\Big(E(\Delta|\Delta| \,|\textbf{X})\Big)=\frac{c_{n,p}^4}{p^2}\Big(\sum_{i=1}^pE(k_{i}^2)+\sum_{i\neq i'}Cov(k_{i},k_{i'})\Big).
	\end{align}
		By (\ref{lz7}), (\ref{lemma1}) and (\ref{rij2}), we get
	\begin{align}
	\notag	Ek^2_i&\leq E\Big(\sum_{j=1\atop j\neq i
	}^{p}\Big(r^2_{ij}-r^2_{i^*j}\Big)\Big)^4\\
	\notag&\leq C\Big\{ E\sum_{j=1\atop j\neq i
	}^{p}\Big(r^2_{ij}-r^2_{i^*j}\Big)^4 +E\sum_{j_1\neq	j_2\neq i}\Big(r^2_{ij_1}-r^2_{i^*j_1}\Big)^2\Big(r^2_{ij_2}-r^2_{i^*j_2}\Big)^2 \Big \}
	\\
	&=O\Big(\frac{1}{n^2}\Big).\label{l33}
	\end{align}
	For $Cov(k_{i},k_{i'})$, we have
	\begin{align*}
	Cov(k_{i},k_{i'})&=Cov\Big(k_i,k_{i'}^{(i)}\Big)+Cov\Big(k_{i'},k_{i'}^{(i)}\Big)\\
	&+Cov\Big(k_{i'}^{(i)},k_{i}^{(i')}\Big)+Cov\Big(k_{i}-k_{i}^{(i')},k_{i'}-k_{i'}^{(i)}\Big)\\
	&=Cov\Big(k_{i}-k_{i}^{(i')},k_{i'}-k_{i'}^{(i)}\Big).
	\end{align*}
	Then we get
	$$E\Big(k_{i}-k_{i}^{(i')}\Big)\Big(k_{i'}-k_{i'}^{(i)}\Big)\leq \frac{1}{2}\Big(E\Big(k_{i}-k_{i}^{(i')}\Big)^2+E\Big(k_{i'}-k_{i'}^{(i)}\Big)^2\Big). $$

	By (\ref{lz7}), we obtain
	\begin{align}
\notag	E\Big(k_{i}-k_{i}^{(i')}\Big)^2&\leq     8E\Big((r_{ii'}^2-r_{i^*i'}^2)^4+(r_{ii'}^2-r_{i^*i'}^2)^2\Big(\sum_{j=1\atop j\neq i,i'}^p(r_{ij}^2-r_{i^*j}^2)\Big)^2\Big)\\
\notag	&=E\Big((r_{ii'}^2-r_{i^*i'}^2)^4+(r_{ii'}^2-r_{i^*i'}^2)^2\Big(\sum_{j=1\atop j\neq i,i'}^p(r_{ij}^2-r_{i^*j}^2)^2\Big) \Big)\\
\notag	&\leq E(r^2_{ii'}-r_{i^*i'}^2)^4+\sum^p_{j=1\atop j\neq i,i'}E\Big((r_{ii'}^2-r_{i^*i'}^2)^2\cdot(r_{ij}^2-r_{i^*j}^2)^2\Big) \\
\notag    &\leq C\Big(E\Big(r^2_{ii'}-\frac{1}{n-1}\Big)^4+\sum_{j=1\atop j\neq i,i'}^pE[(r_{ii'}^2-r_{i^*i'}^2)^2|X_{i},X_{i^*}]E[(r_{ij}^2-r_{i^*j}^2)^2|X_{i},X_{i^*}]\Big)\\
    &=O(\frac{1}{n^{3}}).\label{l32}
	\end{align}
		Thus, by (\ref{l31}), (\ref{l33}) and (\ref{l32}), the proof is complete.
\end{proof}

Now, by (\ref{rij2}), we notice that
\begin{align*}
EW_{n,p}^4=&c^4_{n,p}\Big\{ \sum_{i=2}^{p}\sum_{j=1}^{i-1} \Big(r^2_{ij}-\frac{1}{n-1}\Big)^4 +\sum E\Big(r_{i_1j_1}^2-\frac{1}{n-1}\Big)^2\cdot\Big(r_{i_2j_2}^2-\frac{1}{n-1}\Big)^2\\
&+
\sum E\Big(r^2_{i_1j_1}-\frac{1}{n-1}\Big)\Big(r^2_{i_1j_2}-\frac{1}{n-1}\Big)\Big(r^2_{i_2j_1}-\frac{1}{n-1}\Big)\Big(r^2_{i_2j_2}-\frac{1}{n-1}\Big)\Big\}.
\end{align*}
And we know the facts
\begin{align*}
&E\Big(r^2_{ij}-\frac{1}{n-1}\Big)^4=O\Big(\frac{1}{n^2}\Big),\\
&E\Big(r_{i_1j_1}^2-\frac{1}{n-1}\Big)^2\cdot\Big(r_{i_2j_2}^2-\frac{1}{n-1}\Big)^2=O\Big(\frac{1}{n^4}\Big),\\
&E\Big(r^2_{i_1j_1}-\frac{1}{n-1}\Big)\Big(r^2_{i_1j_2}-\frac{1}{n-1}\Big)\Big(r^2_{i_2j_1}-\frac{1}{n-1}\Big)\Big(r^2_{i_2j_2}-\frac{1}{n-1}\Big)=O\Big(\frac{1}{n^5}\Big).
\end{align*}
Combining those relations above, we can conclude that $ EW_{n,p}^4\leq C $. By the same arguments in the proof of (\ref{th4.1} ) , the bound  can be improved by replacing $ \frac{C}{1+|z|} $ with $ \frac{C}{(1+|z|)^2} $. Finally, by (\ref{2.3}), (\ref{l21}) and (\ref{l3}), we  establish (\ref{itr}). 

{\bf Acknowledgments}

We thank Prof. Qiman Shao and Dr. Zhuosong Zhang for helpful comments and suggestions and thank Prof. Jordan Stoyanov for his suggestions on revision. This research work is supported by the National Natural Science Foundation of China (No.  11701331),  Shandong Provincial Natural Science Foundation (No. ZR2017QA007) and Young Scholars Program of Shandong University.

\section*{Reference}

\bibliographystyle{amsplain}

\end{document}